\documentclass[twoside]{article}
\usepackage{graphics}
\usepackage{amscd,bbm,dsfont,amsmath,amssymb,mathrsfs,amsthm,epsfig,graphics}
\makeindex
\theoremstyle{plain}
\newtheorem{thm}{Theorem}[section]
\newtheorem{lem}{Lemma}[section]
\newtheorem{prop}{Proposition}[section]
\newtheorem{cor}{Corollary}[section]
\newtheorem{defn}[thm]{Definition}
\theoremstyle{remark}
\newtheorem{rem}{Remark}[section]
\numberwithin{equation}{section}

 \DeclareMathOperator{\dist}{dist}
 \DeclareMathOperator{\p}{\partial}
 \DeclareMathOperator{\ls}{\leqslant}
 \DeclareMathOperator{\gs}{\geqslant}

 \newcommand{\e}{\epsilon}
 \newcommand{\ve}{\varepsilon}

\newcommand{\bra}[1]{\left(#1\right)}
\newcommand{\bbra}[1]{\big(#1\big)}
\newcommand{\Bbra}[1]{\Big(#1\Big)}
\newcommand{\squ}[1]{\left[#1\right]}
\newcommand{\bsqu}[1]{\big[#1\big]}
\newcommand{\Bsqu}[1]{\Big[#1\Big]}
\newcommand{\abs}[1]{\left\vert#1\right\vert}
\newcommand{\babs}[1]{\big\vert#1\big\vert}

\newcommand{\set}[1]{\left\{#1\right\}}
\newcommand{\bset}[1]{\big\{#1\big\}}
\newcommand{\Bset}[1]{\Big\{#1\Big\}}
\newcommand{\seq}[1]{\left<#1\right>}
\newcommand{\bseq}[1]{\big<#1\big>}

\newcommand{\norm}[1]{\left\Vert#1\right\Vert}

\def\XXint#1#2#3{{\setbox0=\hbox{$#1{#2#3}{\int}$}
     \vcenter{\hbox{$#2#3$}}\kern-.5\wd0}}

\begin{document}
\title{\vspace{-1in}\parbox{\linewidth}{\footnotesize\noindent
}\vspace{\bigskipamount} \\
On sharp lower bound of the spectral gap for a Schr\"odinger operator and some related results
%
\thanks{{A Project Funded by the Priority Academic Program
Development of Jiangsu Higher Education Institutions;}
{Partially supported by NNSF grant of P. R. China: No.\,11171158.}
\hfil\break\indent
{\em Mathematics Subject Classifications 2010:}
Primary 35P15, 58C40, 65N25, 35J05, 58J05, 35B50.\hfil\break\indent
{\em Key words:} Schr\"odinger operator, Laplace operator, Spectral gap,
Ground state, Strictly convex domain, Diameter of domain.\hfil\break\indent
}}
\date{}
\author{Yue He}
\maketitle

\begin{abstract}
In this paper, we give an easy proof of the main results of Andrews and Clutterbuck's paper
[J. Amer. Math. Soc. 24 (2011), no. 3, 899--916],
which gives both a sharp lower bound for the spectral gap of a Schr\"oinger operator
and a sharp modulus of concavity for the logarithm of the corresponding first eigenfunction.
We arrive directly at same estimates by the `double coordinate' approach
and asymptotic behavior of parabolic flows.
Although using the techniques appeared in the above paper,
we partly simplify the method and argument.
This maybe help to provide an easy way for estimating spectral gap.
Besides, we also get a new lower bound of spectral gap for a class of Sch\"odinger operator.
\end{abstract}

\section{Introduction}
Let $\Omega\subset\mathbb{R}^n$ be a smooth strictly convex bounded domain with boundary
$\p\Omega$ and $V:\bar{\Omega}\mapsto\mathbb{R}$ a convex function.
Consider the following Dirichlet eigenvalue problem of the Schr\"{o}dinger operator
\begin{equation}\label{dep-s-11-6-8-20-58}
\left\{
\begin{array}{l}
-\Delta u+Vu=\lambda u\qquad \hbox{in}\quad\Omega, \\
u=0\qquad \hbox{on}\quad \p \Omega,
\end{array}
\right.
\end{equation}
According to \cite{Courant-Hilbert-1953,Singer-Wong-Yau-Yau-1985,Gilbarg-Trudinger},
problem \eqref{dep-s-11-6-8-20-58} has a countable discrete set of eigenvalues
\[
\{\lambda_i\,\,\big|\,\,\lambda_1<\lambda_2\ls\lambda_3\ls\cdots\nearrow\infty\},
\]
whose eigenfunctions $\{\phi_j\}$ span $W_0^{1,2}(\Omega)$.
Here $\phi_j$ is a normalized eigenfunction corresponding to $\lambda_j$.
In particular, the first eigenfunction $\phi_1$ and its corresponding eigenvalue $\lambda_1$ are called
the ground state and ground state energy, respectively.
The difference between the first two eigenvalues, $\lambda_2-\lambda_1$, is called the fundamental gap
(or `the spectral gap' for short). It is of great significance in quantum mechanics,
statistical mechanics and quantum field theory \cite{Andrews-Clutterbuck-2011-JAMS}.
The spectral gap also determines the rate at which positive solutions of the heat equation
are close in first eigenfunction,
and it is through this characterization that
one can prove the following conjecture \cite{Andrews-Clutterbuck-2011-JAMS}:
\begin{quote}
{\bf Gap Conjecture}.
\emph{Let $\Omega\subset\mathbb{R}^n$ be a bounded strictly convex domain with smooth boundary,
and $V$ a weakly convex potential. Then the eigenvalues of the Schr\"odinger operator satisfy}
\begin{equation}\label{spectral-gap}
\lambda_2-\lambda_1\gs \frac{3\pi^2}{d^2},
\end{equation}
\emph{where $d:=\sup_{x,y\in\Omega}\abs{y-x}$ is the diameter of $\Omega$.}
\end{quote}

We next introduce a notion appears on this paper.
\begin{defn}
Let $\tilde{V}\in C^1\big([-\frac{d}{2},\frac{d}{2}],\mathbb{R}\big)$ be an even function.
we say $\tilde{V}$ is a modulus of expansion for $\nabla V$ if the following inequality holds.
\begin{equation}\label{Modulus-of-convexity}
\bsqu{\nabla V(y)-\nabla V(x)}\cdot\frac{y-x}{\abs{y-x}}
\gs2\tilde{V}'\big(\frac{\abs{x-y}}{2}\big)
\qquad\hbox{for every}\quad x\neq y\in\Omega.
\end{equation}
In this situation, the function $\tilde{V}$ is also called a modulus of convexity of $V$.
\end{defn}
We also consider the Dirichlet eigenvalue problem of one dimensional Schr\"odinger operator:
\begin{equation}\label{1-dim-Schrodinger}
\left\{
\begin{array}{l}
-\tilde{u}''+\tilde{V}\tilde{u}=\tilde{\lambda}\tilde{u}
\qquad\hbox{in}\quad\big(-\frac{d}{2},\frac{d}{2}\big),\\[5pt]
\tilde{u}(-\frac{d}{2})=\tilde{u}(\frac{d}{2})=0,
\end{array}
\right.
\end{equation}
where $\tilde{V}\in C^1\big([-\frac{d}{2},\frac{d}{2}],\mathbb{R}\big)$
is a modulus of convexity of $V$.
Denote the corresponding eigenvalues (resp. eigenfunctions) by adding a tilde,
e.g. $\tilde{\lambda}_i$ and $\tilde{\phi}_i$, $i=1,2,\cdots$.

Now let us temporarily recall some classical results used below. It is well known
(cf. \cite{Courant-Hilbert-1953,Singer-Wong-Yau-Yau-1985,Schoen-Yau-1994,Yau-2003-IPSM}) that
\begin{description}
  \item[(i)]
The first eigenfunction $\phi_1$ is strictly positive in $\Omega$, so $\log\phi_1$ is well-defined in $\Omega$;
  \item[(ii)]
$D_\nu\phi_1\big|_{\p\Omega}=-c$ with $c>0$,
where $\nu$ is the outward normal direction of $\p\Omega$ with respect to $\Omega$;
  \item[(iii)]
Both the first two eigenfunctions $\phi_1$ and $\phi_2$ are smooth on $\bar{\Omega}$,
and the ratio $\phi_2/\phi_1$ can be extended to $\bar{\Omega}$ as a smooth function;
  \item[(iv)]
Direct computation yields that $\phi_2/\phi_1$ satisfies the following equation
\begin{equation}\label{Laplace-with-drift}
\Delta v+2\nabla\log\phi_1\cdot\nabla v=-(\lambda_2-\lambda_1)v\qquad\hbox{in}\quad\Omega.
\end{equation}
So it easily follows from \eqref{Laplace-with-drift} that $\phi_2/\phi_1$ satisfies
the Neumann boundary condition: $D_\nu v=0$ on $\p\Omega$.
In particular, in the case: $n=1$, $\tilde{\phi}_2/\tilde{\phi}_1$ satisfies the following equation
\begin{equation}\label{1-dim-Laplace-with-drift}
\tilde{v}''+2(\log\tilde{\phi}_1)'\tilde{v}'=-(\tilde{\lambda}_2-\tilde{\lambda}_1)\tilde{v}
\qquad\hbox{in}\quad\big(-\frac{d}{2},\frac{d}{2}\big),
\end{equation}
and the boundary condition: $\tilde{v}'(\pm\frac{d}{2})=0$.
\end{description}

Recently, the study of spectral gap, in particular, of the above gap conjecture,
have received a growing attention with many progresses since the early 1980s.
Historically, this conjecture motivated many related studies. Here we provide
a brief overview of that topic and outline just some of the important work carried out.

First let us take two special examples for Schr\"{o}dinger operator (cf. \cite{Henrot-2006}) here.
One is the case: $n=1$ and $\tilde{V}\equiv\textrm{const.}$. It is well-known that
its Dirichlet eigenvalues and corresponding eigenfunctions
on the interval like $(0,d)$, are as follows
\[
\tilde{\lambda}_k=\frac{(k\pi)^2}{d^2}+\tilde{V}\qquad\hbox{and}\qquad
\tilde{\phi}_k=\sin\frac{k\pi x}{d}\qquad\qquad (k\in\mathbb{N})
\]
respectively. Clearly, $\tilde{\lambda}_1=\pi^2/d^2+\tilde{V}$ and $\tilde{\lambda}_2=4\pi^2/d^2+\tilde{V}$, and
this gap is $3\pi^2/d^2$.
The other is the case: $n=2$ and $V\equiv\textrm{const.}$. Similarly,
its Dirichlet eigenvalues and corresponding eigenfunctions
on the rectangle $(0,a)\times(0,b)$, are:
\[
\lambda_{k,l}=\pi^2\bbra{\frac{k^2}{a^2}+\frac{l^2}{b^2}}+V\qquad\hbox{and}\qquad
\phi_{k,l}=\sin\bbra{\frac{k\pi x}{a}}\sin\bbra{\frac{l\pi y}{b}}\qquad\qquad (k,l\in\mathbb{N})
\]
respectively. Note that
\[
\lambda_1=\pi^2\bbra{\frac{1}{a^2}+\frac{1}{b^2}}+V\qquad\hbox{and}\qquad
\lambda_2=\pi^2\bbra{\frac{4}{a^2}+\frac{1}{b^2}}+V.
\]
Thus,
\[
\lambda_2-\lambda_1=\frac{3\pi^2}{a^2}\overset{or}=\frac{3\pi^2}{d^2-b^2}\gs\frac{3\pi^2}{d^2}.
\]
So, that is the reason $3\pi^2/d^2$ is usually regarded as sharp lower bound of the spectral gap.

It was observed by M. van den Berg (1983) \cite{van-den-Berg-1983} that $\lambda_2-\lambda_1\gs3\pi^2/d^2$
holds for many convex domains. This was also independently suggested by Yau (1986) \cite{Yau-1986-Monographies},
as well as Ashbaugh and Benguria (1989) \cite{Ashbaugh-Benguria-1989}, which is just the so-called fundamental gap conjecture above.

In the one dimension case: Ashbaugh and Benguria (1989) \cite{Ashbaugh-Benguria-1989} asserted that
the fundamental gap conjecture holds if $\tilde{V}$ is single-well and symmetric (not necessarily convex).
To some extent Horv\'{a}th (2003) \cite{Horvath-2003} removed the symmetry assumption,
allowing $\tilde{V}$ to be a single-well potential with minimum at the mid-point of the interval.
Lavine (1994) \cite{Lavine-1994} proved that the fundamental gap conjecture holds if $\tilde{V}$ is convex.

In the higher dimensions case: Yau et al (1985) \cite{Singer-Wong-Yau-Yau-1985}
used the fact that the first eigenfunction is logarithmic concave and combined the maximum principle method
originated by Li and Yau (e.g. \cite{Li-1979,Li-Yau-1983}) to obtain that the gap is bounded below by $\pi^2/(4d^2)$.
Later in \cite{Yau-1986-Monographies} (see also Appendix I and II of \cite{Schoen-Yau-1994}),
Yau further conjectured that one could improve the lower bound of the gap to $3\pi^2/d^2$, when $\Omega$ is an interval.

By sharpening Li and Yau's techniques (see e.g. \cite{Li-1979,Li-Yau-1983,Singer-Wong-Yau-Yau-1985}),
Yu and Zhong (1986) \cite{Yu-Zhong-1986} gave a more delicate gradient estimate,
which had been applied to improve the above estimate to $\pi^2/d^2$.

Of course, it is obvious that the optimal lower bound of $\lambda_2-\lambda_1$ would be perfect.
To improve previous results via the maximum principle method only,
one need to construct suitable test functions requiring more detailed technical work.
Ling (2008) \cite{Ling-2008-CAG} provided new estimates that improve
sequentially the lower bound of the gap to $\frac{\pi^2}{d^2}+\frac{31}{50}\alpha$,
with $\alpha=-\sup_\Omega\nabla^2(\log\phi_1)$.

Besides, the gap conjecture has been proved in a special domain case:
Ba\~{n}uelos and M\'{e}ndez-Hern\'{a}ndez (2000) \cite{Banuelos-Hernandez-2000}
(with potential $V=0$), Davis (2001) \cite{Davis-2001}, also Ba\~{n}elos and Kr\"{o}ger (2001)
\cite{Banelos-Kroger-2001} independently achieved that the fundamental gap conjecture holds if
$\Omega\subset \mathbb{R}^2$ is symmetric with respect to $x$ and $y$ axes, and convex in both $x$ and $y$.

As we already know, $\phi_2$ changes sign in $\Omega$.
But the problem is that in general, we do not know whether
$\phi_2$ is of symmetry, i.e. $\sup_{x\in\Omega}\phi_2=-\inf_{x\in\Omega}\phi_2$.
As a result, we also do not know whether $\phi_2/\phi_1$ is always symmetric.
That is why the above maximum principle method does not work well for getting
the optimal estimate of spectral gap. In spite of many efforts to this conjecture,
no effective method is available in the literature before. However, the situation have been changed by
Andrews and Clutterbuck (2011) \cite{Andrews-Clutterbuck-2011-JAMS}, which solve the above gap conjecture
via a differential method. 
Intuitively, their proof is probably considered as the best of all previous arguments.
In particular, their argument avoids any problems arising
from possible asymmetry mentioned above.
More precisely, they first prove an important `comparison' theorem, which
controls the modulus of continuity of solutions of a Neumann heat equation with drift term, in terms of
the modulus of contraction of the drift velocity. Their key step in proving the gap conjecture
is to establish a sharp log-concavity estimate for the first eigenfunction.
By such `comparison' theorem, they deduce that the desired sharp log-concavity estimate above.
Using this sharp estimate,
they achieve that the eigenvalue gap is bounded below by the gap of an associated Sturm-Liouville problem on
a closed interval via the study of the asymptotics to a parabolic problem, while the latter in the
weakly convex potential case is just $3\pi^2/d^2$, In this way the gap conjecture is eventually proved
(see \cite{Andrews-Clutterbuck-2011-JAMS}, Proposition 3.2 and Corollary 1.4 for more details).
In a word, their results were thought to be sharp in both of a lower bound for the spectral gap of
Schr\"oinger operator and a modulus of concavity for the logarithm of the corresponding first
eigenfunction in terms of the diameter of the domain and a modulus of convexity for the potential.

However, to author's acknowledge, there is rarely understanding in case of more general elliptic operator.
Somewhat later, Ni (2013) \cite{Lei-Ni-2013} developed a deep and meaningful analysis firstly
introduced by Andrews et al, for giving an alternate proof of the main results of
\cite{Andrews-Clutterbuck-2011-JAMS}, in the convex potential case. Moreover,
Ni also used his method to solve several problems of eigenvalue estimate.
In fact, Ni's method can be consider as an important generalization (or application)
of \cite{Andrews-Clutterbuck-2011-JAMS} in the analytic setting. More recently,
a similar technique was used in a more difficult case by Wolfson \cite{Wolfson-2012}
for the proof of a eigenvalue gap theorem.
He described how the approach due to \cite{Andrews-Clutterbuck-2011-JAMS,Lei-Ni-2013}
works well by estimating the eigenvalue gap for a class of nonsymmetric second order
linear elliptic operators. Taken together, these papers suggest a general approach
to estimating the eigenvalue gap of a large class of linear second-order elliptic
operators on convex domains.
Besides the above works, we ought to mention that the idea of estimating the spectral gap
using probabilistic methods (in particular, the Kendall--Cranston coupling method) was developed
by Chen and Wang in the early 1990s, see for instance \cite{Chen-Wang-1994,Chen-Wang-1997} and
Chapters 2 to 3  in \cite{Chen-Wang-2005}. Alternatively, Hsu also gave a short introduction
of this method in (\cite{Hsu-2002}, Section 6.7).
Along this direction, Gong et al (2014) \cite{Gong-Liu-Liu-Luo-2014} recently extended the spectral
gap comparison theorem of \cite{Andrews-Clutterbuck-2011-JAMS} to the infinite dimensional setting.
In a more recent paper \cite{Gong-Li-Luo-2013}, Gong et al gave a probabilistic proof to
the fundamental gap conjecture (in the convex potential case) via the coupling by a reflection.
For more information on spectral gap, the reader is usually referred to the original paper
\cite{Andrews-Clutterbuck-2011-JAMS}, and excellent literatures \cite{Lei-Ni-2013,Wolfson-2012,Yau-2003-IPSM,Yau-2008-MC,Gong-Liu-Liu-Luo-2014,Gong-Li-Luo-2013},
also the relevant references therein.
To sum up, with the rapid development of spectral geometry,
the study of that topic on this stage is getting more and more important.

Since the idea and structure of \cite{Lei-Ni-2013,Wolfson-2012} are all some close to
\cite{Andrews-Clutterbuck-2011-JAMS}, we want to find a comparatively easy proof (maybe which
is a little bit independent of \cite{Andrews-Clutterbuck-2011-JAMS}) to the gap conjecture.
In this paper we have two motivations: One is to understand more intuitively those proofs due to
\cite{Andrews-Clutterbuck-2011-JAMS,Lei-Ni-2013,Wolfson-2012}. The other is devoted to
provide a simplified method for solving the gap conjecture. Precisely,
we essentially use a `double coordinate' approach introduced by Clutterbuck (2004)
\cite{Clutterbuck-2004-phD-thesis} from the parabolic flows point of view. Unifying this approach
with the asymptotics of parabolic problem,
we arrive ultimately at same estimates of \cite{Andrews-Clutterbuck-2011-JAMS}.
Further, there is another difference in the reasoning between \cite{Andrews-Clutterbuck-2011-JAMS} and this paper,
that is, we do not use Sturm-Liouville theory of ODEs at all.
Although in some ways analogous to the strategy used in \cite{Andrews-Clutterbuck-2011-JAMS},
our proof looks more directly since we only use an argument based on the theory of parabolic flows. As a result,
we really reduce the method and complexity of argument to some degree.
This maybe help to provide an easy way of estimating spectral gap.
Besides, we also give a new lower bound of spectral gap for a class of Sch\"odinger operator
via a technique used in \cite{Shi-Zhang-2007,Qian-Zhang-Zhu-2012}.

As \cite{Andrews-Clutterbuck-2011-JAMS} previously pointed out,
one key step in estimating the spectral gap is to establish
a sharp log-concavity estimate of ground state.
So, one of our aim is to derive such a sharp estimate.
Let us now state this deeper result (cf. \cite[Theorem 1.5]{Andrews-Clutterbuck-2011-JAMS};
see also \cite{Lei-Ni-2013,Wolfson-2012,Gong-Li-Luo-2013} in convex potential case) as follows.
\begin{thm}\label{sharp-log-concavity-estimate}
Let $\Omega\subset\mathbb{R}^n$ be a bounded strictly convex domain with
smooth boundary. If the functions $V$ and $\tilde{V}$ are related by
\eqref{Modulus-of-convexity}, i.e., $\tilde{V}$ is a modulus of convexity of $V$.
Then the first eigenfunction $\phi_1$ of the Sch\"odinger operator with potential $V$, satisfies
\begin{equation}\label{est-13-11-23-20-56}
\bsqu{\nabla\log\phi_1(y)-\nabla\log\phi_1(x)}\cdot\frac{y-x}{\abs{y-x}}
\ls2(\log\tilde{\phi}_1)'\big(\frac{\abs{y-x}}{2}\big),
\end{equation}
for every $x\neq y$ in $\Omega$, where $\tilde{\phi}_1$ is the first eigenfunction of \eqref{1-dim-Schrodinger}.
\end{thm}
\begin{rem}
Obviously, the estimate \eqref{est-13-11-23-20-56} improves Brascamp and Lieb's result
\cite{Brascamp-Liep-1976} $($see also \cite{Singer-Wong-Yau-Yau-1985}$)$,
which showed that if $V$ is convex, then $\phi_1$ is log-concave, that is, $\nabla^2\log\phi_1\ls0$.
Moreover, \eqref{est-13-11-23-20-56} is also regarded as the sharp log-concavity estimate of ground state.
\end{rem}

Our main interests in this paper focus on the sharp lower bound of
spectral gap, and prove the following well-known comparison result
(see \cite[Theorem 1.3]{Andrews-Clutterbuck-2011-JAMS}).

\begin{thm}\label{spectral-gap-comparison-thm}
Under the same assumptions as Theorem \ref{sharp-log-concavity-estimate}.
Then the eigenvalues of the Sch\"odinger operator with potential $V$, satisfy
\[
\lambda_2-\lambda_1\gs\tilde{\lambda}_2-\tilde{\lambda}_1.
\]
The latter is just the spectral gap of problem \eqref{1-dim-Schrodinger}.
\end{thm}

In many special cases (such as that in \cite{van-den-Berg-1983,Ashbaugh-Benguria-1989,Horvath-2003,Lavine-1994}, etc.),
it happens that $\tilde{\lambda}_2-\tilde{\lambda}_1$ is $3\pi^2/d^2$. Therefore,
we are not difficult to obtain the following conclusion.
\begin{cor}\label{spectral-gap-cor}
Under the same assumptions as Theorem \ref{sharp-log-concavity-estimate},
but $\tilde{V}(s)$ also satisfies certain conditions so that the spectral gap
$\tilde{\lambda}_2-\tilde{\lambda}_1$ of problem \eqref{1-dim-Schrodinger}
is just $3\pi^2/d^2$. Then the eigenvalues of the Sch\"odinger operator
with potential $V$, satisfy
\[
\lambda_2-\lambda_1\gs\frac{3\pi^2}{d^2}.
\]
\end{cor}

\begin{rem}
As far as we know, `weakly convex potential' $V$ in the paper \cite{Andrews-Clutterbuck-2011-JAMS},
is essentially a generalized convex function possessing weak derivatives in some sense.
As a result, Corollary \ref{spectral-gap-cor} obviously implies that
the previous gap conjecture is true.
\end{rem}

As an application of Theorem \ref{spectral-gap-comparison-thm},
we argue as \cite{Shi-Zhang-2007,Qian-Zhang-Zhu-2012} and obtain the following result:
\begin{cor}\label{cor-new-esti-of-gap}
Assume everything is as in Theorem \ref{sharp-log-concavity-estimate}. Then
the eigenvalues of the Sch\"odinger operator with potential $V$, satisfy
\begin{equation}\label{esti-14-5-26-11-16}
\lambda_2-\lambda_1\gs4s(1-s)\frac{\pi^2}{d^2}
+2s\tilde{\alpha}\qquad\hbox{for all}\quad s\in (0,1),
\end{equation}
where $\tilde{\alpha}:=-\sup_{\tau\in(-\frac{d}{2},\frac{d}{2})}\big(\log\tilde{\phi}_1\big)''(\tau)$.
In particular, if let $s=1/2$, then the above estimate becomes
\begin{equation}\label{est-14-6-17-19-20}
\lambda_2-\lambda_1\gs\frac{\pi^2}{d^2}+\tilde{\alpha}.
\end{equation}
\end{cor}

\begin{rem}
It seems that \eqref{est-14-6-17-19-20} improves Ling's result \cite{Ling-2008-CAG}:
$\lambda_2-\lambda_1\gs\frac{\pi^2}{d^2}+\frac{31}{50}\alpha$, with $\alpha=-\sup_\Omega\nabla^2(\log\phi_1)$.
\end{rem}


Throughout this paper we always assume that $\Omega$ has smooth boundary and is uniformly convex
and that $V$ is smooth. The results for the general case of convex $\Omega$ and $V$ are still valid,
and follow using a straightforward approximation argument as described in
\cite{Henrot-Pierre-2005} (see also \cite[Theorem 2.3.17]{Henrot-2006}).

The rest of this paper is organized as follows. In Section 2, we fix notations, and 
establish necessary lemmas and formulas needed below. In Section 3,
we intend to give a direct elementary proof of Theorem \ref{sharp-log-concavity-estimate}
via the `double coordinate' approach and asymptotic behavior of a parabolic flow.
With Theorem \ref{sharp-log-concavity-estimate} in hand, by the similar technique as in 
Theorem \ref{sharp-log-concavity-estimate}, we prove Theorem \ref{spectral-gap-comparison-thm} in Section 4.
In section 5, we discuss a special example concerning the fundamental gap conjecture.
As an application of Theorem \ref{spectral-gap-comparison-thm}, we also offer another
lower bound of the spectral gap for the problem \eqref{dep-s-11-6-8-20-58}.
For the reader's convenience, we give some important properties
(used in other parts of this paper) of two special functions in Appendix A.
Finally, in order to prove Theorem \ref{sharp-log-concavity-estimate} in Section 3,
we also need to construct an auxiliary function via a distance function in Appendix B.

\section{Notations and preliminaries}
Throughout this paper, we shall exclusively use the following notations (cf. \cite{Lu-Seminar-May-26-2011}):
\[
X:=\abs{y-x},
\]
where $\abs{\cdot}$ denotes the Euclidean norm on $\mathbb{R}^n$;
\[
X_i:=\frac{\p X}{\p y_i}\overset{\textrm{or}}
=-\frac{\p X}{\p x_i}=\frac{y_i-x_i}{\abs{y-x}}\qquad\hbox{for}\quad i=1,\cdots,n;
\]
\[
X_{ij}:=\frac{\p X_i}{\p y_j}\overset{\textrm{or}}=-\frac{\p X_i}{\p x_j}
=\frac{1}{X}(\delta_{ij}-X_iX_j)\qquad\hbox{for}\quad i,j=1,\cdots,n;
\]
\begin{eqnarray}\label{eq-13a-1-5-10-39}
X_{ijk}&:=&\frac{\p X_{ij}}{\p y_k}\overset{\textrm{or}}=-\frac{\p X_{ij}}{\p x_k}\nonumber\\
&=&\frac{1}{X}(-X_{ik}X_j-X_iX_{jk})-\frac{1}{X^2}(\delta_{ij}-X_iX_j)X_k\nonumber\\
&=&-\frac{1}{X}(X_{ik}X_j+X_iX_{jk}+X_{ij}X_k);
\end{eqnarray}
and
\[
\tilde{X}:=\nabla_yX\overset{\textrm{or}}=-\nabla_xX=\frac{y-x}{\abs{y-x}}=(X_1,\cdots,X_n).
\]
It is not hard to verify that
\[
\frac{\p}{\p x_i}X=-X_i,\qquad\frac{\p}{\p y_j}X=X_j,
\qquad\frac{\p^2X}{\p x_i\p y_j}=-X_{ij},
\]
and
\[
\frac{\p}{\p x_k}X_{ij}=-X_{ijk},\qquad\frac{\p}{\p y_k}X_{ij}=X_{ijk}.
\]

Some simple properties about the above terms can be summarized as follows:
\[
\sum_i(X_i)^2=1;
\]
\[
X_{ij}=X_{ji},\qquad\hbox{for}\quad\forall\,\,i,j=1,\cdots,n;
\]
\[
X_{ijk}=X_{jki}=X_{kij},\qquad\hbox{for}\quad\forall\,\,i,j,k=1,\cdots,n;
\]
and
\begin{equation}\label{eq-13a-1-5-10-51}
\sum_iX_iX_{ij}=0,\qquad\hbox{for}\quad\forall\,\,j=1,\cdots,n.
\end{equation}
Indeed, noting $\sum_i(X_i)^2=1,$ we can check the last identity in the following:
\begin{eqnarray*}
\sum_iX_iX_{ij}&=&\frac{1}{X}\sum_iX_i\bra{\delta_{ij}-X_iX_j}\\
&=&\frac{1}{X}\Big\{X_j-\sum_i(X_i)^2\cdot X_j\Big\}\\
&=&\frac{1}{X}(X_j-1\cdot X_j)=0.
\end{eqnarray*}

Let us define an coupling operator as follows (cf. \cite{Lu-Seminar-May-26-2011}):
\begin{eqnarray}\label{def-13a-1-13a-23-21}
L&:=&\sum_i\Big(\frac{\p}{\p x_i}+\frac{\p}{\p y_i}\Big)^2-4\sum_{i,j}X_iX_j\frac{\p^2}{\p x_i\p y_j}\nonumber\\
&\overset{\textrm{or}}=&\Delta_x+\Delta_y+2\sum_i\frac{\p^2}{\p x_i\p y_i}-4\sum_{i,j}X_iX_j\frac{\p^2}{\p x_i\p y_j}.
\end{eqnarray}
on $\Omega\times\Omega\subset \mathbb{R}^{2n}$.

In order to prove Theorem \ref{sharp-log-concavity-estimate} and \ref{spectral-gap-comparison-thm},
we next state two lemmas (see \cite{Lu-Seminar-May-26-2011}).
For the reader's convenience, we also give the corresponding proofs below.
\begin{lem}
\begin{description}
  \item[(i)]
The matrix $(X_{ij})_{n\times n}$ is semi-positive definite;
  \item[(ii)]
The coupling operator $L$ is degenerate elliptic;
  \item[(iii)]
Let $h$ be a smooth single variable function. Then
\begin{equation}\label{eq-13-12-7-18-31}
L\bbra{h(X)}=4h''(X).
\end{equation}
\end{description}
\end{lem}
\begin{proof}
(i)\,For any $\xi\in\mathbb{R}^n$, by Cauchy--Schwarz inequality, we have
\begin{eqnarray*}
\sum_{i,j}X_{ij}\xi_i\xi_j&=&\sum_{i,j}\frac{1}{X}(\delta_{ij}-X_iX_j)\xi_i\xi_j\\
&=&\frac{1}{X}\Big\{\sum_i\xi_i^2-\sum_iX_i\xi_i\cdot\sum_jX_j\xi_j\Big\}\\
&=&\frac{1}{X}\Big\{\sum_i\xi_i^2-\bbra{\sum_iX_i\xi_i}^2\Big\}\\
&\gs&\frac{1}{X}\Big\{\sum_i\xi_i^2-\sum_iX_i^2\cdot\sum_i\xi_i^2\Big\}=0.
\end{eqnarray*}
Therefore, the conclusion is true.

(ii)\,We only need to show that the corresponding quadratic form of $L$ is nonnegative.
That is, the following inequality holds for any $\xi$, $\eta\in\mathbb{R}^n$.
\[
\sum_i(\xi_i+\eta_i)^2-4\sum_{i,j}X_iX_j\xi_i\eta_j\gs0,
\]
which implies the conclusion. Let $\seq{\cdot}$ denote the inner product on $\mathbb{R}^n$.
In fact, using $\bseq{\tilde{X},\tilde{X}}=|\tilde{X}|^2=1$, we have
\begin{eqnarray*}
&&\sum_i(\xi_i+\eta_i)^2-4\sum_{i,j}X_iX_j\xi_i\eta_j\\
&=&\abs{\xi+\eta}^2-4\bseq{\xi,\tilde{X}}\bseq{\eta,\tilde{X}}\\
&=&\abs{\xi+\eta}^2-\big(\bseq{\xi+\eta,\tilde{X}}^2-\bseq{\xi-\eta,\tilde{X}}^2\big)\\
&=&\abs{\xi+\eta}^2-2\bseq{\xi+\eta,\tilde{X}}^2
+\bseq{\xi+\eta,\tilde{X}}^2\bseq{\tilde{X},\tilde{X}}+\bseq{\xi-\eta,\tilde{X}}^2\\
&=&\babs{(\xi+\eta)-\bseq{\xi+\eta,\tilde{X}}\tilde{X}}^2+\bseq{\xi-\eta,\tilde{X}}^2\gs0,
\end{eqnarray*}
and with equality if and only if
\[
\left\{
\begin{array}{ll}
\xi+\eta=\bseq{\xi+\eta,\tilde{X}}\tilde{X},\\[5pt]
\bseq{\xi-\eta,\tilde{X}}=0,
\end{array}
\right.
\]
which is equivalent to $\abs{\xi}=\abs{\eta}$.

(iii)\,By directly calculating, it is easy to verify that
\[
\frac{\p^2h(X)}{\p x_i^2}=h''(X)(X_i)^2+h'(X)X_{ii},\qquad
\frac{\p^2h(X)}{\p y_i^2}=h''(X)(X_i)^2+h'(X)X_{ii},
\]
\[
\frac{\p^2h(X)}{\p x_i\p y_i}=-h''(X)(X_i)^2-h'(X)X_{ii},\qquad
\frac{\p^2h(X)}{\p x_i\p y_j}=-h''(X)X_iX_j-h'(X)X_{ij}.
\]

Applying the above identities and \eqref{eq-13a-1-5-10-51}, we obtain
\begin{eqnarray*}
L\bbra{h(X)}&=&h''(X)\Big\{\sum_i\bsqu{(X_i)^2+(X_i)^2-2(X_i)^2}+4\sum_{i,j}X_iX_jX_iX_j\Big\}\\
&&+h'(X)\Big[\sum_i\bbra{X_{ii}+X_{ii}-2X_{ii}}+4\sum_{i,j}X_iX_jX_{ij}\Big]\\
&=&4h''(X)\sum_{i,j}X_iX_jX_iX_j+4h'(X)\sum_{i,j}X_iX_jX_{ij}\\
&=&4h''(X)\Bsqu{\sum_i(X_i)^2}^2+4h'(X)\sum_iX_i\Bbra{\sum_jX_jX_{ij}}\\
&=&4h''(X).
\end{eqnarray*}
Thus we get the desired equality \eqref{eq-13-12-7-18-31}. So far the lemma is proved. 
\end{proof}

For simplicity, let us simply adopt the following notations:
\[
v_i(x)=\frac{\p v(x)}{\p x_i},\qquad v_{ij}(x)=\frac{\p^2v(x)}{\p x_i\p x_j},
\qquad v_{ijk}(x)=\frac{\p^3v(x)}{\p x_i\p x_j\p x_k},
\]
and so on.

\begin{lem}\label{lem-13-11-28-8-3}
For any smooth function $v$, we have
\begin{equation}\label{eq-13-11-27-11-9}
L\bset{\bsqu{\nabla v(y)-\nabla v(x)}\cdot\tilde{X}}
=\bsqu{\nabla\Delta v(y)-\nabla\Delta v(x)}\cdot\tilde{X},
\end{equation}
or, equivalently,
\[
L\Bset{\sum_i\bsqu{v_i(y)-v_i(x)}X_i}=\sum_{i,j}\bsqu{v_{ijj}(y)-v_{ijj}(x)}X_i.
\]
\end{lem}
\begin{proof}
Denote 
\[
W:=\bsqu{\nabla v(y)-\nabla v(x)}\cdot\tilde{X}
\overset{\textrm{or}}=\sum_i\bsqu{v_i(y)-v_i(x)}X_i.
\]
We first compute as follows:
\begin{eqnarray*}
&&\Big(\frac{\p}{\p x_j}+\frac{\p}{\p y_j}\Big)^2\Big\{\bsqu{v_i(y)-v_i(x)}X_i\Big\}\\
&=&\Big(\frac{\p}{\p x_j}+\frac{\p}{\p y_j}\Big)\Big\{\bsqu{v_{ij}(y)-v_{ij}(x)}X_i+\bsqu{v_i(y)-v_i(x)}(-X_{ij}+X_{ij})\Big\}\\
&=&\Big(\frac{\p}{\p x_j}+\frac{\p}{\p y_j}\Big)\Big\{\bsqu{v_{ij}(y)-v_{ij}(x)}X_i\Big\}\\
&=&\bsqu{v_{ijj}(y)-v_{ijj}(x)}X_i+\bsqu{v_{ij}(y)-v_{ij}(x)}(-X_{ij}+X_{ij})\\
&=&\bsqu{v_{ijj}(y)-v_{ijj}(x)}X_i.
\end{eqnarray*}
Thus,
\begin{equation}\label{eq-13a-1-6-11-38}
\sum_j\Big(\frac{\p}{\p x_j}+\frac{\p}{\p y_j}\Big)^2W=\sum_{i,j}\bsqu{v_{ijj}(y)-v_{ijj}(x)}X_i.
\end{equation}

On the other hand,
\[
\sum_{j,k}X_jX_k\frac{\p^2W}{\p x_j\p y_k}
=-\sum_{i,j,k}X_jX_k\Big\{v_{ik}X_{ij}+v_{ij}X_{ik}+\bsqu{v_i(y)-v_i(x)}X_{ijk}\Big\}.
\]
Using \eqref{eq-13a-1-5-10-51}, we get
\[
\sum_{i,j,k}X_jX_kv_{ki}(y)X_{ij}=\sum_{i,k}X_kv_{ki}(y)\sum_jX_jX_{ij}=0.
\]
Similarly,
\[
\sum_{i,j,k}X_jX_kv_{ij}(x)X_{ik}=0.
\]

In addition, from \eqref{eq-13a-1-5-10-39} and \eqref{eq-13a-1-5-10-51}, it follows that
\begin{eqnarray*}
&&\sum_{i,j,k}X_jX_k\bsqu{v_i(y)-v_i(x)}X_{ijk}\\
&=&-\frac{1}{X}\sum_{i,j,k}X_jX_k\bsqu{v_i(y)-v_i(x)}(X_{ik}X_j+X_iX_{jk}+X_{ij}X_k)\\
&=&\sum_i\bsqu{v_i(y)-v_i(x)}\Big\{\bbra{\sum_kX_kX_{ik}}\cdot\sum_j(X_j)^2\\
&&\qquad+\sum_jX_i X_j\bbra{\sum_kX_kX_{jk}}+\bbra{\sum_jX_jX_{ij}}\cdot\sum_k(X_k)^2\Big\}\\
&=&0.
\end{eqnarray*}
Thus,
\begin{equation}\label{eq-13a-1-6-11-53}
\sum_{i,j}X_iX_j\frac{\p^2W}{\p x_i\p y_j}=0.
\end{equation}
So we derive the desired identity \eqref{eq-13-11-27-11-9} from \eqref{eq-13a-1-6-11-38}
and \eqref{eq-13a-1-6-11-53}. This completes the proof.
\end{proof}

\section{On sharp log-concavity estimate of ground state}
Before proving Theorem \ref{sharp-log-concavity-estimate} we precede it with a well-known
result, sometimes called approximation lemma (see \cite[Lemma 2.1]{Lee-Vazquez}).
For our purposes, it is technically simpler that an alternative approach to
Theorem \ref{sharp-log-concavity-estimate} based on this lemma.

\begin{lem}\label{approximation-lemma}
$($\textrm{Approximation lemma}$)$. Let $u=u(x,t)$ be the solution of the problem
\begin{equation}\label{heat-flow-13-12-3-21-21}
\left\{
  \begin{array}{lll}
u_t=\Delta u-Vu\qquad\hbox{in}\quad\Omega\times(0,T],\\
u(x,0)=u_0(x),\\
u=0\qquad\hbox{on}\quad\p\Omega\times[0,T].
  \end{array}
\right.
\end{equation}
For every $u_0\in L^2(\Omega)$ we have
\begin{equation}\label{Approximation-1}
\abs{e^{\lambda_1t}u(x,t)-a_1\phi_1(x)}\ls Ce^{-(\lambda_2-\lambda_1)t},
\end{equation}
and
\begin{equation}\label{Approximation-2}
\norm{e^{\lambda_1t}u(x,t)-a_1\phi_1(x)}_{C_x^k(\Omega)}
\ls Ce^{-(\lambda_2-\lambda_1)t}\qquad\hbox{for all}\quad k=1,2,\cdots,
\end{equation}
where $a_1=\int_\Omega u_0(x)\phi_1(x)$, and the constant $C$ does not depend on $T$.
\end{lem}
Use the same argument as that in \cite{Lee-Vazquez},
we can prove the desired conclusion as follows.
\begin{proof}

It is obvious that $e^{-\lambda_jt}\phi_j(x)$ is the solution of the equation $u_t=\Delta u-Vu$
with initial data $\phi_j(x)$. On the other hand, for $u_0\in L^2(\Omega)$ there are coefficients
$\{a_n\}$ such that $u_0(x)=\sum_{j=1}^\infty a_j\phi_j(x)$. Hence, the solution $u(x,t)$ of the problem
\eqref{heat-flow-13-12-3-21-21} can be written as follows
\[
u(x,t)=\sum_{j=1}^\infty a_je^{-\lambda_jt}\phi_j(x)
=a_1e^{-\lambda_1t}\phi_1(x)+e^{-\lambda_2t}\eta(x,t),
\]
where $\eta(x,t)=\sum_{j=2}^\infty a_je^{-(\lambda_j-\lambda_2)t}\phi_j(x)$, and
\[
\norm{\eta}_{L_x^2(\Omega)}^2\ls\sum_{j=2}^\infty a_j^2\norm{\phi_j}_{L^2(\Omega)}^2
=\sum_{j=2}^\infty a_j^2\ls\norm{u_0}_{L^2(\Omega)}^2.
\]
Thus, we get
\begin{equation}\label{eqn-14-6-20-9-22}
e^{\lambda_1t}u(x,t)-a_1\phi_1(x)=e^{-(\lambda_2-\lambda_1)t}\eta(x,t).
\end{equation}

It is easy to check that $\eta(x,t)$ should satisfy the equation
$\eta_t=\Delta\eta-V\eta+\lambda_2\eta$ in $\Omega\times(0,T]$ and the conditions:
\[
\left\{
  \begin{array}{ll}
\eta(x,0)=u_0-a_1\phi_1(x),\\[3pt]
\eta=0\quad\hbox{on}\quad\p\Omega\times[0,T].
  \end{array}
\right.
\]
So, the $L^2$-boundedness of $\eta$ and standard parabolic estimates
(see \cite[Theorem 5.14, 4.28 and 7.17]{Lieberman-1996};
or \cite[Theorem 3.5, 3.4 and 3.11]{Bei-Hu-2011};
see also \cite[Theorem 4.2]{Ya-Zhe-Chen-2003}, etc.) tell us that
\[
\norm{\eta}_{L^\infty},\,\,\norm{\eta}_{C_x^k(\Omega)}\ls C(\delta_0)
\qquad\hbox{for all}\quad t\gs\delta_0\quad\hbox{and}\quad k=1,2,\cdots,
\]
from which and \eqref{eqn-14-6-20-9-22}, we can deduce the conclusions of Lemma.
Since $\eta=0$ on $\p\Omega\times[0,T]$, and the above equation is linearly homogeneous,
with $V$ independent of the variable $t$, hence the constant $C(\delta_0)$
appeared in the previous estimates does not depend on $T$.
\end{proof}
Now let us give an easy proof of Theorem \ref{sharp-log-concavity-estimate}
via a `double coordinate' approach (cf. \cite{Clutterbuck-2004-phD-thesis,Andrews-Clutterbuck-2011-JAMS})
and the above approximation lemma. We are partly consulting
\cite{Andrews-Clutterbuck-2011-JAMS,Lu-Seminar-May-26-2011} in this process of proof.
Note that there is no essentially different between \cite{Andrews-Clutterbuck-2011-JAMS} and here,
as we will prove Theorem \ref{sharp-log-concavity-estimate} specially, the coupling operator $L$ is
exactly the same as the one used in the original proof (see \cite{Andrews-Clutterbuck-2011-JAMS}).
Because some of its details, such as Theorem \ref{spectral-gap-comparison-thm}, are used below.
So a detailed account of each of these steps follows.

\underline{Proof of Theorem \ref{sharp-log-concavity-estimate}}.
\begin{proof}
We first choose a function $u_0(x)\in W_0^{1,2}(\Omega)$,
such that $u_0(x)>0$ in $\Omega$, and
\begin{equation}\label{Phi-log-concavity-u0}
\bsqu{\nabla\log u_0(y)-\nabla\log u_0(x)}\cdot\tilde{X}\ls\Phi(X),
\qquad\hbox{whenever}\quad x\neq y\quad\hbox{in}\quad\Omega,
\end{equation}
while
\begin{equation}\label{defn-Phi-func}
\Phi(s):=2(\log\tilde{\phi}_1)'(s/2)\overset{or}=2(\tilde{\phi}_1'/\tilde{\phi}_1)(s/2)
\qquad\hbox{for all}\quad s\in[0,d].
\end{equation}
Because some tedious manipulation is need to construct $u_0(x)$, we hence leave this highly technical
task in Appendix B (see \eqref{initial-data-function} below).

We next consider the problem \eqref{heat-flow-13-12-3-21-21}, and then letting $T\rightarrow+\infty$.
Let $u$ denote the solution of the problem \eqref{heat-flow-13-12-3-21-21}.
It is easy to verify that $v:=\log u$ satisfies the equation:
\begin{equation}\label{eqn-v-14-6-22-23-11}
v_t=\Delta v+\abs{\nabla v}^2-V.
\end{equation}

Define an evolving quantity
\begin{eqnarray*}
Z_\ve(x,y,t)&:=&\bsqu{\nabla v(y,t)-\nabla v(x,t)}\cdot\tilde{X}-\Phi(X)-\ve e^{Ct}\\
&\overset{or}=&\sum_i\bsqu{v_i(y,t_0)-v_i(x,t_0)}X_i-\Phi(X)-\ve e^{Ct},
\end{eqnarray*}
for small $\ve>0$ and some suitably large $C$ to be chosen (independent of $\ve$).

We first show that $Z_\ve(x,y,0)<0$ for all $x,\,y\in\Omega$ with $x\neq y$.
This is achieved by the above choice of $u_0(x)$.

We next claim that $Z_\ve(x,y,t)<0$ for all $x,\,y\in\Omega$ with $x\neq y$ and all $t>0$.
We argue by contradiction and assume that $Z_\ve$ at $(x_0,y_0,t_0)$, vanishes for the first time.
Clearly, $t_0>0$, $x_0\neq y_0\in\Omega$.

For simplicity, we drop the subscripts in $x_0$ and $y_0$ if there
is no confusion and still denote them by $x$ and $y$ respectively.

Since $Z_\ve(x,y,t_0)=0$,
then
\[
\sum_i\bsqu{v_i(y,t_0)-v_i(x,t_0)}X_i=\Phi(X)+\ve e^{Ct_0}.
\]
Moreover, the maximum principle concludes that
\[
D_tZ_\ve\gs0,\qquad\nabla_xZ_\ve=\nabla_yZ_\ve=0,\qquad LZ_\ve\ls0\qquad\hbox{at}\quad(x,y,t_0).
\]
Clearly, $\nabla_xZ_\ve(x,y,t_0)=\nabla_yZ_\ve(x,y,t_0)=0$, implies that
\begin{equation*}\label{e-13-11-28-8-22}
\left\{
\begin{array}{ll}
\displaystyle-\sum_iv_{ij}(x,t_0)X_i-\bsqu{v_i(y,t_0)-v_i(x,t_0)}X_{ij}+\Phi'(X)X_j=0,\\[10pt]
\displaystyle\sum_iv_{ij}(y,t_0)X_i+\bsqu{v_i(y,t_0)-v_i(x,t_0)}X_{ij}-\Phi'(X)X_j=0,
\end{array}
\right.
\end{equation*}
for all $i,j=1,\cdots,n$. Multiplying the first equality and second equality
by $v_j(x,t_0)$ and $v_j(y,t_0)$ respectively, summing over $j$ successively,
and also take into account that the matrix $(X_{ij})_{n\times n}$ is semi-positive definite, we have
\begin{eqnarray}\label{ineq-13-11-28-10-26}
&&\sum_{i,j}\bsqu{v_j(y,t_0)v_{ij}(y,t_0)-v_j(x,t_0)v_{ij}(x,t_0)}X_i\nonumber\\
&=&-\sum_{i,j}\bsqu{v_j(y,t_0)-v_j(x,t_0)}\cdot\bsqu{v_i(y,t_0)-v_i(x,t_0)}X_{ij}
+\Phi(X)\Phi'(X)+\ve\Phi'(X)e^{Ct_0}\nonumber\\
&\ls&\Phi(X)\Phi'(X)+\ve\Phi'(X)e^{Ct_0}.
\end{eqnarray}

From \eqref{eq-13-12-7-18-31} and \eqref{eq-13-11-27-11-9}, it is very easy to verify that
\[
LZ_\ve(x,y,t_0)=\bsqu{\nabla\Delta v(y,t_0)-\nabla\Delta v(x,t_0)}\cdot\tilde{X}-4\Phi''(X).
\]
Thus
\begin{equation}\label{eq-13-11-28-10-21}
\bsqu{\nabla\Delta v(y,t_0)-\nabla\Delta v(x,t_0)}\cdot\tilde{X}=LZ_\ve(x,y,t_0)+4\Phi''(X).
\end{equation}

Using Eq. \eqref{eqn-v-14-6-22-23-11} satisfied by $v$, we proceed by dealing with $Z_\ve$ as follows
\begin{eqnarray*}
D_tZ_\ve(x,y,t_0)&=&\bsqu{\nabla v_t(y,t_0)-\nabla v_t(x,t_0)}\cdot\tilde{X}-C\ve e^{Ct_0}\\
&=&\bsqu{\nabla\Delta v(y,t_0)-\nabla\Delta v(x,t_0)}\cdot\tilde{X}\\
&&+\big[\nabla\bbra{\abs{\nabla v(y,t_0)}^2}-\nabla\bbra{\abs{\nabla v(x,t_0)}^2}\big]\cdot\tilde{X}\\
&&-\bsqu{\nabla V(y)-\nabla V(x)}\cdot\tilde{X}-C\ve e^{Ct_0}\\
&=&\bsqu{\nabla\Delta v(y,t_0)-\nabla\Delta v(x,t_0)}\cdot\tilde{X}\\
&&+2\sum_{i,j}\bsqu{v_j(y,t_0)v_{ji}(y,t_0)-v_j(x,t_0)v_{ji}(x,t_0)}X_i\\
&&-\bsqu{\nabla V(y)-\nabla V(x)}\cdot\tilde{X}-C\ve e^{Ct_0}.
\end{eqnarray*}
Combining \eqref{ineq-13-11-28-10-26} with \eqref{eq-13-11-28-10-21} we conclude that at $(x,y,t_0)$,
\begin{eqnarray*}
D_tZ_\ve&\ls&LZ_\ve+4\Phi''(X)+2\Phi(X)\Phi'(X)+2\ve\Phi'(X)e^{Ct_0}\\
&&-\bsqu{\nabla V(y)-\nabla V(x)}\cdot\tilde{X}-C\ve e^{Ct_0}.
\end{eqnarray*}
By a simple calculation, we have
\begin{eqnarray}\label{eqn-14-5-23a-13-22}
\Phi'(s)&=&-\frac{\tilde{\phi}_1''(\frac{s}{2})\tilde{\phi}_1(\frac{s}{2})-[\tilde{\phi}_1'(\frac{s}{2})]^2}
{[\tilde{\phi}_1(\frac{s}{2})]^2}
=\frac{\tilde{\phi}_1''(\frac{s}{2})}{\tilde{\phi}_1(\frac{s}{2})}
-\Big[\frac{\tilde{\phi}_1'(\frac{s}{2})}{\tilde{\phi}_1(\frac{s}{2})}\Big]^2\nonumber\\
&=&\frac{\tilde{V}(\frac{s}{2})\tilde{\phi}_1(\frac{s}{2})-\tilde{\lambda}_1\tilde{\phi}_1(\frac{s}{2})}
{\tilde{\phi}_1(\frac{s}{2})}-\frac{1}{4}[\Phi(s)]^2
=\tilde{V}(\frac{s}{2})-\tilde{\lambda}_1-\frac{1}{4}[\Phi(s)]^2,
\end{eqnarray}
and
\[
\Phi''(s)=\frac{1}{2}\tilde{V}'(\frac{s}{2})-\frac{1}{2}\Phi(s)\Phi'(s).
\]
From this we easily get that
\[
4\Phi''(X)+2\Phi(X)\Phi'(X)=2\tilde{V}'\big(\frac{X}{2}\big).
\]
In addition, \eqref{Modulus-of-convexity} implies that
\[
-\bsqu{\nabla V(y)-\nabla V(x)}\cdot\tilde{X}
\ls-2\tilde{V}'\big(\frac{\abs{y-x}}{2}\big)\overset{or}=-2\tilde{V}'\big(\frac{X}{2}\big).
\]
Since $L$ is a degenerate elliptic operator, then $LZ_\ve(x,y,t_0)\ls0$.
Thus we conclude that
\[
0\ls D_tZ_\ve\ls LZ_\ve+\big[2\Phi'(X)-C\big]\ve e^{t_0}\ls\big[2\Phi'(X)-C\big]\ve e^{t_0}<0,\qquad\hbox{at}\quad(x,y,t_0).
\]
The last strict inequality is produced by choosing
\[
C:=2\|\tilde{V}\|_{C([0,\frac{d}{2}])}\gs2\tilde{V}\big(\frac{s}{2}\big)
>2\Big[\tilde{V}\big(\frac{s}{2}\big)-\tilde{\lambda}_1
-\big(\frac{\tilde{\phi}'_1}{\tilde{\phi}_1}\big)^2\big(\frac{s}{2}\big)\Big]
=2\Phi'(s).
\]
Clearly, $C$ is independent of $\ve$ as required. This really leads
to a contradiction in assuming that $Z_\ve$ does not remain negative.
So, $Z_\ve(x,y,t)<0$ always holds for all
$(x,y,t)\in\Omega\times\Omega\times\mathbb{R}^+$.
Therefore, sending $\ve\rightarrow0$, we deduce that
\[
Z(x,y,t):=\bsqu{\nabla v(y,t)-\nabla v(x,t)}\cdot\tilde{X}-\Phi(X)
=\lim_{\ve\rightarrow0}Z_\ve(x,y,t)\ls0.
\]

Since $u(x,t)$ is a solution of  \eqref{heat-flow-13-12-3-21-21},
then Lemma \ref{approximation-lemma} yields that
\[
\lim_{t\rightarrow\infty}e^{\lambda_1t}u(x,t)=a_1\phi_1(x)\qquad\hbox{and}\qquad
\lim_{t\rightarrow\infty}e^{\lambda_1t}\nabla u(x,t)=a_1\nabla\phi_1(x),
\]
while
\[
a_1=\int_\Omega u_0(x)\phi_1(x)>0,
\]
since $u_0>0$ in $\Omega$. Thus
\begin{eqnarray*}
\lim_{t\rightarrow\infty}\nabla v(x,t)
&=&\lim_{t\rightarrow\infty}\nabla\log u(x,t)
=\lim_{t\rightarrow\infty}\frac{\nabla u(x,t)}{u(x,t)}
\displaystyle=\frac{\lim_{t\rightarrow\infty}e^{\lambda_1t}\nabla u(x,t)}
{\lim_{t\rightarrow\infty}e^{\lambda_1t}u(x,t)}\\[10pt]
&=&\frac{a_1\nabla\phi_1(x)}{a_1\phi_1(x)}
=\frac{\nabla\phi_1(x)}{\phi_1(x)}=\nabla\log\phi_1(x).
\end{eqnarray*}
Finally, note that $Z(x,y,t)\ls0$ for all $t\gs0$, so we deduce that
\[
\bsqu{\nabla\log\phi_1(y)-\nabla\log\phi_1(x)}\cdot X-\Phi(X)
=\lim_{t\rightarrow\infty}Z(x,y,t)\ls0,
\]
from which the desired estimate \eqref{est-13-11-23-20-56} follows easily and the theorem is proved.
\end{proof}

\section{Proof of Theorem \ref{spectral-gap-comparison-thm}}
All the preparations have been completed, so we are ready to prove Theorem \ref{spectral-gap-comparison-thm}
via a `double coordinate' approach (cf. \cite{Clutterbuck-2004-phD-thesis,Andrews-Clutterbuck-2011-JAMS}) .
In this process we will use some argument similar to but not exactly like
\cite{Andrews-Clutterbuck-2011-JAMS,Lu-Seminar-May-26-2011}.
\begin{proof}
Noticing that \eqref{lower-bound-esti-Psi} (in Appendix A) and Lipschitz continuity of
$\phi_2/\phi_1$, and choosing some suitably small $\delta_0>0$, we can easily show that
\begin{equation}\label{defn-14-6-20-22-36}
\delta_0\cdot\frac{\phi_2(y)}{\phi_1(y)}-\delta_0\cdot\frac{\phi_2(x)}{\phi_1(x)}
\ls\Psi(X),\quad\forall\,\,x,y\in\Omega,
\end{equation}
with $\Psi(s):=(\tilde{\phi}_2/\tilde{\phi}_1)(s/2)$ for all $0\ls s\ls d$.

Put
\[
v(x,t):=\delta_0\cdot\frac{\phi_2(x)}{\phi_1(x)}e^{-(\lambda_2-\lambda_1)t}
\overset{or}=\delta_0\cdot\frac{\phi_2(x)e^{-\lambda_2t}}{\phi_1(x)e^{-\lambda_1t}}.
\]
Obviously, $\phi_1(x)e^{-\lambda_1t}$ and $\phi_2(x)e^{-\lambda_2t}$ are
two smooth solutions of the parabolic Sch\"odinger equation
\[
u_t=\Delta u-Vu\qquad\hbox{in}\quad\Omega\times\mathbb{R}^+,\\
\]
\[
u=0\qquad\hbox{on}\quad\p\Omega\times\mathbb{R}^+.
\]
According to Proposition 3.1 in \cite{Andrews-Clutterbuck-2011-JAMS},
we know that $v$ is smooth on $\Omega\times\mathbb{R}^+$, and satisfies
the following Neumann heat equation with drift:
\begin{equation}\label{Neumann-heat-flow}
v_t=\Delta v+2\nabla\log\phi_1\cdot\nabla v\qquad\hbox{in}\quad\Omega\times[0,\infty),\\[5pt]
\end{equation}
\[
D_\nu v=0\qquad\hbox{on}\quad\p\Omega\times[0,\infty),
\]
where $\nu$ is defined as before.

Set $\sigma:=\tilde{\lambda}_2-\tilde{\lambda}_1$, we next define an evolving quantity
\[
Z_\ve(x,y,t):=v(y,t)-v(x,t)-e^{-\sigma t}\Psi(X)-\ve e^t,
\]

By virtue of \eqref{def-13a-1-13a-23-21} and \eqref{eq-13-12-7-18-31}, we can deduce that
\begin{eqnarray*}
LZ_\ve(x,y,t)&=&Lv(y,t)-Lv(x,t)-e^{-\sigma t}L\bbra{\Psi(X)}\\[5pt]
&=&\Delta v(y,t)-\Delta v(x,t)-4e^{-\sigma t}\Psi''(X).
\end{eqnarray*}
In view of the definition of $v(x,t)$ and \eqref{defn-14-6-20-22-36},
we know that $Z_\ve(x,y,0)<0$. We now claim that
\begin{equation}\label{less-than-0-13a-1-5}
Z_\ve(x,y,t)<0\qquad\hbox{for all}\quad(x,y)\in\bar{\Omega}\times\bar{\Omega},\,\,t\gs0.
\end{equation}
As in the proof of Theorem \ref{sharp-log-concavity-estimate}, we derive the assertion by contradiction.
If this is not true,
then there exists the first time, $t_0>0$, and points
$x_0\neq y_0\in\bar{\Omega}$ such that $Z_\ve(x_0,y_0,t_0)=0$. Thus,
\[
Z_\ve(x,y,t)<0\qquad\hbox{for all}\quad(x,y)\in\bar{\Omega}\times\bar{\Omega},\,\,\,0\ls t<t_0.
\]
For simplicity, we still denote $x_0$ and $y_0$ by $x$ and $y$ respectively.

There are two possibilities: Either both $x$ and $y$ are in the interior of $\Omega$,
or at least one of them lies in the boundary.

Now we consider the first case first: $x,\,y\in\Omega$.
According to the maximum principle, the maximality of $Z_\ve$ implies that at $(x,y,t_0)$,
\[
\left\{
  \begin{array}{lll}
D_tZ_\ve(x,y,t_0)\gs0,\\[5pt]
\nabla_xZ_\ve(x,y,t_0)=\nabla_yZ_\ve(x,y,t_0)=0,\\[5pt]
LZ_\ve(x,y,t_0)\ls0.
  \end{array}
\right.
\]
A direct calculation leads to
\begin{equation}\label{13-11-23-21-6}
\left\{
  \begin{array}{llll}
v_t(y,t_0)-v_t(x,t_0)+\sigma e^{-\sigma t_0}\Psi(X)-\ve e^t\gs0,\\[5pt]
-\nabla v(x,t_0)+e^{-\sigma t_0}\Psi'(X)\cdot\tilde{X}=0,\\[5pt]
\nabla v(y,t_0)-e^{-\sigma t_0}\Psi'(X)\cdot\tilde{X}=0,\\[5pt]
\Delta v(y,t_0)-\Delta v(x,t_0)-4e^{-\sigma t_0}\Psi''(X)\ls0.
  \end{array}
\right.
\end{equation}
Taking the dot product with $\nabla\log\phi_1(x)$ in the second equality,
with $\nabla\log\phi_1(y)$ in the third equality respectively, and summing successively, we have
\begin{eqnarray}\label{13-11-23-21-10}
&&\nabla\log\phi_1(y)\cdot\nabla v(y,t_0)-\nabla\log\phi_1(x)\cdot\nabla v(x,t_0)\\
&=&e^{-\sigma t_0}\Psi'(X)\bsqu{\nabla\log\phi_1(y)-\nabla\log\phi_1(x)}\cdot\tilde{X}.\nonumber
\end{eqnarray}

Noting that $\Psi'(X)>0$ and applying \eqref{est-13-11-23-20-56},
\eqref{13-11-23-21-6} and \eqref{13-11-23-21-10}, we have
\begin{eqnarray*}
0&\ls&v_t(y,t_0)-v_t(x,t_0)+\sigma e^{-\sigma t_0}\Psi(X)-\ve e^t\\
&=&\bsqu{\Delta v(y,t_0)-\Delta v(x,t_0)}
+2\squ{\nabla\log\phi_1(y)\nabla v(y,t_0)-\nabla\log\phi_1(x)\nabla v(x,t_0)}\\
&&+\sigma e^{-\sigma t_0}\Psi(X)-\ve e^t\\
&=&4e^{-\sigma t_0}\Psi''(X)
+e^{-\sigma t_0}\Psi'(X)
\bsqu{\nabla\log\phi_1(y)-\nabla\log\phi_1(x)}\cdot\tilde{X}\\
&&+\sigma e^{-\sigma t_0}\Psi(X)-\ve e^t\\
&\ls&4e^{-\sigma t_0}\Psi''(X)
+2e^{-\sigma t_0}\Psi'(X)\cdot(\log\tilde{\phi}_1)'\big(\frac{X}{2}\big)
+\sigma e^{-\sigma t_0}\Psi(X)-\ve e^t\\
&=&e^{-\sigma t_0}\Big[4\Psi''(X)+2\Psi'(X)\cdot(\log\tilde{\phi}_1)'\big(\frac{X}{2}\big)
+\sigma\Psi(X)\Big]-\ve e^t\\
&=&-\ve e^t<0.
\end{eqnarray*}
Here the last equality is due to \eqref{1-dim-Laplace-with-drift}. This is obviously a contradiction.

We next deal with the second case: $x$ or $y\in\p\Omega$.

(i) If $y\in\p\Omega$ and $0<X:=\abs{y-x}<d$, then 
\[
D_{\nu_y}Z_\ve(x,y,t_0)=D_{\nu_y}v(y,t_0)
-e^{-\sigma t_0}\Psi'(X)\cdot\tilde{X}\cdot\nu_y<0,
\]
where $\nu_y$ is the outer unit normal to $\Omega$ at $y$.
Here we used the Neumann boundary condition satisfies by $v$, $\Psi'(X)>0$ and
the strict convexity of $\Omega$ (which implies $\tilde{X}\cdot\nu_y>0$).
This implies $Z_\ve(x,y-s\nu_y,t_0)>0$ for $s$ small enough, contradicting the fact that
$Z_\ve(x,y,t)\ls0$ on $\bar{\Omega}\times\bar{\Omega}\times[0,t_0]$.
The case where $x\in\p\Omega$ and $0<X<d$ is similar.

(ii) If $x,\,y\in\p\Omega$ and $X=d$, then $D_{\nu_y}Z_\ve(x,y,t_0)=0$ since
$D_{\nu_y}v(y,t_0)=0$ and $\Psi'(X)=0$. On the other hand, it is easy to see that
all the tangential derivatives in $y$ of $Z_\ve$ vanish at $(x,y,t_0)$.
So, we get that $\nabla_yZ_\ve(x,y,t_0)=0$. Similarly, we can prove that $\nabla_xZ_\ve(x,y,t_0)=0$.
The remaining derivation proceeds exactly as that in the case where $x,\,y\in\Omega$.

So far, we have arrived at a contradiction if we are assuming that
$Z_\ve$ does not remain negative. Therefore, \eqref{less-than-0-13a-1-5}
always holds for all $x,\,y\in\bar{\Omega}$ and all $t\gs0$.

We next consider the long time behavior of solutions of \eqref{Neumann-heat-flow}.
By \eqref{less-than-0-13a-1-5}, we get that
\[
v(y,t)-v(x,t)-e^{-\sigma t}\Psi(X)
=\lim_{\ve\rightarrow0}Z_\ve(x,y,t)\ls0
\]
holds for all $(x,y,t)\in\bar{\Omega}\times\bar{\Omega}\times\mathbb{R}$. Thus,
\[
v(y,t)-v(x,t)\ls e^{-\sigma t}\Psi(X)\qquad\hbox{for all}\quad
(x,y,t)\in\bar{\Omega}\times\bar{\Omega}\times\mathbb{R},
\]
or, equivalently:
\[
\delta_0\cdot\frac{\phi_2(y)}{\phi_1(y)}-\delta_0\cdot\frac{\phi_2(x)}{\phi_1(x)}
\ls e^{-[\sigma-(\lambda_2-\lambda_1)t]}\Psi(X)
\qquad\hbox{for all}\quad(x,y,t)\in\bar{\Omega}\times\bar{\Omega}\times\mathbb{R}
\]
Once this is done, we conclude that $\sigma\ls\lambda_2-\lambda_1$. Otherwise,
suppose $\sigma>\lambda_2-\lambda_1$. It follows that
\[
\delta_0\cdot\frac{\phi_2(y)}{\phi_1(y)}-\delta_0\cdot\frac{\phi_2(x)}{\phi_1(x)}
\ls\Psi(X)\cdot\lim_{t\rightarrow\infty}e^{-[\sigma-(\lambda_2-\lambda_1]t}=0.
\]
So, we get
\[
\frac{\phi_2(y)}{\phi_1(y)}\ls\frac{\phi_2(x)}{\phi_1(x)}.
\]
Owing to the choice of $x$ and $y$ are quite arbitrary, we know easily that
\[
\frac{\phi_2(x)}{\phi_1(x)}\equiv\textrm{const.},
\]
which gives a contradiction. Therefore,
$\lambda_2-\lambda_1\gs\sigma$ as required.
Theorem is proved.
\end{proof}

\section{Examples and application}
Assume that $\tilde{V}(s)$ satisfies certain conditions so that the spectral gap
$\tilde{\lambda}_2-\tilde{\lambda}_1$ of problem \eqref{1-dim-Schrodinger}
is just $3\pi^2/d^2$ (cf., \cite{van-den-Berg-1983,Ashbaugh-Benguria-1989,Horvath-2003,Lavine-1994}).
In this situation, Theorem \ref{spectral-gap-comparison-thm} gives Corollary \ref{spectral-gap-cor}.
For instance, when $V$ is convex, $\tilde{V}=0$ is apparently a modulus of convexity for $V$.
Clearly, $\tilde{\lambda}_2-\tilde{\lambda}_1=3\pi^2/d^2$, hence $\lambda_2-\lambda_1\gs3\pi^2/d^2$.
Here we refer to \cite{Andrews-Clutterbuck-2011-JAMS} for more examples concerning the fundamental gap conjecture.

Next we give an application of Theorem \ref{spectral-gap-comparison-thm}.
Following the similar argument in \cite{Shi-Zhang-2007,Qian-Zhang-Zhu-2012},
we get that the following results:
\begin{lem}\label{lem-new-esti-of-gap}
The spectral gap $\tilde{\lambda}_2-\tilde{\lambda}_1$
of the problem \eqref{1-dim-Schrodinger}, satisfies
\begin{equation}\label{1-dim-new-esti-of-gap}
\tilde{\lambda}_2-\tilde{\lambda}_1\gs4s(1-s)\frac{\pi^2}{d^2}+2s\tilde{\alpha},
\quad\forall\,\,\,s\in (0,1),
\end{equation}
where $\tilde{\alpha}:=-\sup_{\tau\in(-\frac{d}{2},\frac{d}{2})}\big(\log\tilde{\phi}_1\big)''(\tau)$.
Moreover, we also have
\begin{equation}\label{1-dim-new-esti-of-2nd-eigen}
\tilde{\lambda}_2\gs(1+2s)\tilde{\lambda}_1+4s(1-s)\frac{\pi^2}{d^2}
+2s\inf_{(-\frac{d}{2},\frac{d}{2})}\big[(\tilde{\phi}'_1/\tilde{\phi}_1)^2-\tilde{V}\big],
\quad\forall\,\,\,s\in (0,1).
\end{equation}
\end{lem}
\begin{proof}
By differentiating \eqref{1-dim-Laplace-with-drift}, it leads to
\begin{equation}\label{3-order-ode-eqn-1-dim}
\tilde{v}'''+2(\log\tilde{\phi}_1)'\tilde{v}''
+\big[2(\log\tilde{\phi}_1)''+\tilde{\lambda}_2-\tilde{\lambda}_1\big]\tilde{v}'=0
\qquad\hbox{in}\quad\big(-\frac{d}{2},\frac{d}{2}\big).
\end{equation}

For any constant $a>1$, let $s:=1-\frac{1}{a}$. Clearly, $0<s<1$.
Multiplying \eqref{3-order-ode-eqn-1-dim} by $(\tilde{v}')^{a-1}$
and integrating over $(-\frac{d}{2}, \frac{d}{2})$, and then using
integration by parts and also the condition $\tilde{v}'(\pm\frac{d}{2})=0$, it follows that
\begin{eqnarray*}
&&4s(1-s)\int_{-\frac{d}{2}}^\frac{d}{2}\Big\{\big[(\tilde{v}')^\frac{a}{2}\big]'\Big\}^2dx\\
&=&\int_{-\frac{d}{2}}^\frac{d}{2}
\big[2s(\log\tilde{\phi}_1)''+\tilde{\lambda}_2-\tilde{\lambda}_1\big]\big[(\tilde{v}')^\frac{a}{2}\big]^2dx\\
&\ls&\big(-2s\tilde{\alpha}+\tilde{\lambda}_2-\tilde{\lambda}_1\big)
\int_{-\frac{d}{2}}^\frac{d}{2}\big[(\tilde{v}')^\frac{a}{2}\big]^2dx.
\end{eqnarray*}
Finally, since $(\tilde{v}')^\frac{a}{2}(\pm\frac{d}{2})=0$, then Wirtinger's inequality asserts that
\[
4s(1-s)\frac{\pi^2}{d^2}\ls
-2s\tilde{\alpha}+\tilde{\lambda}_2-\tilde{\lambda}_1,
\quad\forall\,\,\,s\in (0,1),
\]
from which \eqref{1-dim-new-esti-of-gap} follows at once.

Furthermore, by a direct calculation we also obtain
\[
(\log\tilde{\phi}_1)''=\tilde{V}-\tilde{\lambda}_1-(\tilde{\phi}'_1/\tilde{\phi}_1)^2.
\]
Substituting this into \eqref{1-dim-new-esti-of-gap},
it is no difficult to get \eqref{1-dim-new-esti-of-2nd-eigen}. The proof is now complete.
\end{proof}

As a consequence, we can derive from Theorem \ref{spectral-gap-comparison-thm} and Lemma
\ref{lem-new-esti-of-gap} to Corollary \ref{cor-new-esti-of-gap}.

\section{Appendix A}
For the reader's convenience, we list some important properties of
functions $\Phi$ and $\Psi$ and provide some elementary proofs below.
These properties have already been used or will be used in other parts of this paper.

\begin{prop}\label{1st-eigenfunc-est}
Let $\tilde{\phi}_1$ be the first eigenfunction of problem \eqref{1-dim-Schrodinger}.
Then the function $\Phi(s)$, which is defined by \eqref{defn-Phi-func}, satisfies that
\begin{equation}\label{phi-1st-eigenfunc-est}
-c_*-\frac{4}{d-s}\ls\Phi(s)\ls c_*-\frac{4}{d-s},\quad\forall\,\,\,0\ls s<d,
\end{equation}
with the constant $c_*$ depends only on $\tilde{\phi}_1$.
\end{prop}
\begin{proof}
Since $\tilde{\phi}_1(\frac{d}{2})=0$ and $\tilde{\phi}_1(t)>0$
for any $0\ls t<\frac{d}{2}$, then by Eq. \eqref{1-dim-Schrodinger}
we know that $\tilde{\phi}_1'(\frac{d}{2})<0$.
By Malgrange Theorem and $\tilde{\phi}_1'(\frac{d}{2})<0$, we know that
\begin{equation}\label{tilde-phi-1}
\tilde{\phi}_1(t)=(\frac{d}{2}-t)f(t),\quad\forall\,\,\,0\ls t\ls\frac{d}{2},
\end{equation}
where $f(t): [0,\frac{d}{2}]\mapsto\mathbb{R}$ is a smooth function,
which satisfies that $f(t)>0$ for any $0\ls t\ls \frac{d}{2}$. Thus
\[
\Phi(s)=2(\frac{\tilde{\phi}'_1}{\tilde{\phi}_1})(\frac{s}{2})
=2\frac{f'(\frac{s}{2})}{f(\frac{s}{2})}-\frac{4}{d-s},\quad\forall\,\,\,0\ls s<d.
\]
Taking $c_*:=2\norm{f'/f}_{C([0,\frac{d}{2}])}$,
the above equality implies the conclusion.
\end{proof}
\begin{prop}\label{Psi-properity}
Let $\tilde{\phi}_1$ and $\tilde{\phi}_2$ be the first two
eigenfunctions of problem \eqref{1-dim-Schrodinger}, respectively.
Set $\Psi(s):=(\tilde{\phi}_2/\tilde{\phi}_1)(s/2)$ for all $s\in[0,d)$.
Then $\Psi$ can be extended to $[0,d]$ as a smooth function, and also satisfies that
\begin{equation}\label{Psi-properity-incr-conca}
\Psi'(d)=0,\quad\hbox{and}\quad\Psi'(s)>0,\quad\Psi''(s)<0\qquad\hbox{for all}\quad0\ls s<d.
\end{equation}
Moreover, we also have
\begin{equation}\label{lower-bound-esti-Psi}
\Psi(s)\gs\bar{c}s\quad\hbox{for all}\quad0\ls s\ls d,
\end{equation}
with $\bar{c}:=\Psi(d)/d>0$.
\end{prop}
\begin{proof}
Clearly, $\Psi$ can be extended to $[0,d]$ as a smooth function.
It directly derives from Eq. \eqref{Laplace-with-drift} and the condition
$(\tilde{\phi}_2/\tilde{\phi}_1)'(\frac{d}{2})=0$ that
\[
\big(\frac{\tilde{\phi}_2}{\tilde{\phi}_1}\big)'(t)
=\frac{\sigma}{\big[\tilde{\phi}_1(t)\big]^2}
\int_t^{\frac{d}{2}}\tilde{\phi}_1(\tau)\tilde{\phi}_2(\tau)\,\mathrm{d}\tau
\qquad\hbox{for all}\quad0\ls t\ls\frac{d}{2},
\]
and
\[
\big(\frac{\tilde{\phi}_2}{\tilde{\phi}_1}\big)''(t)
=(-\sigma)\Big\{\frac{2}{[\tilde{\phi}_1(t)]^3}
\int_t^{\frac{d}{2}}\tilde{\phi}_1(\tau)\tilde{\phi}_2(\tau)\,\mathrm{d}\tau
+\frac{\tilde{\phi}_2(t)}{\tilde{\phi}_1(t)}\Big\}<0
\quad\hbox{for all}\quad0\ls t\ls\frac{d}{2}.
\]
In addition, since $\tilde{\phi}_1(\tau)>0$ and $\tilde{\phi}_2(\tau)>0$ for all
$0<\tau<\frac{d}{2}$, then \eqref{Psi-properity-incr-conca} obviously holds.
Furthermore, $\Psi(0)=0$ since $\tilde{\phi}_2(0)=0$.
Note that $\Psi(0)=0$ and \eqref{Psi-properity-incr-conca} really implying that
$\Psi(s)$ is strictly increasing and concave on $[0,d]$, so we easily get that
$\Psi(s)\gs\bar{c}s$. The Proposition follows.
\end{proof}
\begin{rem}
Andrews et al \cite[the proof of Proposition 3.2]{Andrews-Clutterbuck-2011-JAMS}
had already pointed out: $\Psi(s)$ is positive on $(0,d)$ and has the positive first-order
derivative at $s=0$, so is bounded below by $\bar{c}s$ for some small $\bar{c}>0$.
\end{rem}

\section{Appendix B}
In this section, we will construct an auxiliary function satisfying the inequality \eqref{Phi-log-concavity-u0}.

Let $\Omega\subset\mathbb{R}^n$ be a bounded open domain with smooth boundary $\p\Omega$. Define a distance function
\[
\rho(x):=\dist(x,\p\Omega)=\inf_{z\in\p\Omega}\abs{x-z}\qquad\hbox{for all}\quad x\in\bar{\Omega}.
\]
Clearly, $0\ls\rho(x)\ls \frac{d}{2}$ on $\bar{\Omega}$, and $\rho(x)=0$ if and only if $x\in\p\Omega$.
For $\e>0$, define a domain
\[
\Gamma_\e:=\set{x\in\bar{\Omega};\,\,\rho(x)\ls\e}.
\]
Denote by $\vec{n}$ the unit inward normal vector field on $\p\Omega$. According to \S 14.6 Appendix  
in \cite{Gilbarg-Trudinger}, we can find a sufficient small $\e_*=\e_*(n,\Omega)>0$ so that $\rho$ is smooth on
$\Gamma_{\e_*}$, and for any $x\in\Gamma_{\e_*}$, there exists a unique $\bar{x}\in\p\Omega$
such that
\[
\rho(x)=\abs{x-\bar{x}}\qquad\hbox{and}\qquad\nabla\rho(x)=\vec{n}(\bar{x}).
\]
In particular, $\nabla\rho\equiv\vec{n}$ on $\p\Omega$.
Obviously, $\abs{\nabla\rho}\equiv1$ on $\Gamma_{\e_*}$.

\begin{lem}\label{lem-13-12-31-56}
Let $\Omega\subset\mathbb{R}^n$ be a bounded strictly convex domain
with smooth boundary $\p\Omega$ and diameter $d=d(\Omega)$.
Then there exist two constants $\e_0=\e_0(n,\Omega)$
and $\theta_0=\theta_0(n,\Omega)>0$, such that
\[
\nabla\rho(x)\cdot\frac{y-x}{\abs{y-x}}\gs\theta_0>0,
\]
whenever $x\in\Gamma_{\e_0}$ and $y\in\bar{\Omega}\setminus B_{\frac{d}{2}}(x)$.
\end{lem}
\begin{figure}
\centering
\includegraphics[width=6.5cm,height=4cm]{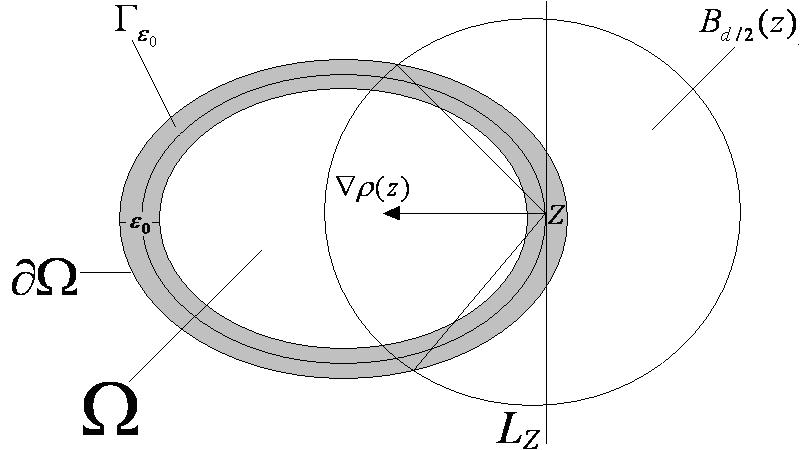}\\
\caption{e.g., $\Omega\subset\mathbb{R}^2$ is bounded strictly convex.}\label{figure}
\end{figure}
\begin{proof}
Since $\Omega$ has smooth boundary and is uniformly convex,
then for any $y\in\p\Omega$, there exists an inscribed sphere $B_y\subset\Omega$
such that $B_y\cap\p\Omega=\{y\}$. Denote by $R_y$ the radius of inscribed sphere $B_y$.
Set $R_0:=\min_{y\in\p\Omega}R_y$.
For $z\in\Gamma_{\e_*}$, let $L_z$ denote a hyperplane, which passes through $z$
and is perpendicular to $\nabla\rho(z)$. By the uniformly convexity of $\Omega$ again,
we know that there exists a constant $\e_0\ls\min\{\e_*,R_0,\frac{d}{2}\}$, such that
$L_z$ is tangent to the hypersurface $\{x\in\Omega\,\big|\,\rho(x)=\rho(z)\}$ at $z$, and
\begin{equation}\label{ineq-13a-1-1-7-16}
L_z\cap\bar{\Omega}\setminus B_{\frac{d}{2}}(z)=\emptyset,
\end{equation}
whenever $z\in\Gamma_{\e_0}$.
For $x\in\Gamma_{\e_0}$, we define a function
\[
\theta(x):=\min\Big\{\nabla\rho(x)\cdot\frac{y-x}{\abs{y-x}}\,:
\quad y\in\bar{\Omega}\setminus B_{\frac{d}{2}}(x)\Big\}
\]
Clearly, $\theta(x)$ is continuous on $\Gamma_{\e_0}$ and $\theta(x)>0$
whenever $x\in\Gamma_{\e_0}$. Set $\theta_0:=\min_{x\in\Gamma_{\e_0}}\theta(x)$.
First, $\theta_0\gs0$ is obvious. We now claim that $\theta_0>0$. Otherwise, suppose $\theta_0=0$.
Then the continuity of $\theta(x)$ yields that
there exist two points $x_0\in\Gamma_{\e_0}$
and $y_0\in\bar{\Omega}\setminus B_{\frac{d}{2}}(x_0)$ such that
\[
\nabla\rho(x_0)\cdot\frac{y_0-x_0}{\abs{y_0-x_0}}=0.
\]
But this contradicts \eqref{ineq-13a-1-1-7-16}. So, we have
\[
\theta(x)\gs\theta_0>0\qquad\hbox{for all}\quad x\in\Gamma_{\e_0},
\]
which implies the conclusion.
\end{proof}

\begin{prop}\label{prop-13a-1-2-12-1}
Let $\Omega$ and $d$ be as in Lemma \ref{lem-13-12-31-56}.
Then there exist two constants $\kappa=\kappa(n,\Omega)$ and
$c_0=c_0(n,\Omega,\tilde{V},\tilde{\phi}_1)$ such that the following inequality
\begin{equation}\label{ineq-13-12-18-22-2}
\Big[\nabla\log\rho^\kappa(y)-\nabla\log\rho^\kappa(x)\Big]\cdot\tilde{X}\ls\Phi(X)+c_0X,
\end{equation}
holds whenever $x\neq y$ in $\Omega$. Here $\Phi(s)$ is defined as before.
\end{prop}
\begin{proof}
Take $\e_1:=\min\{\e_0,4/(c_*+1)\}$, with the same constant
$c_*$ in proposition \ref{1st-eigenfunc-est}.

We must consider two possibilities: If $X>d-\e_1$, then
\[
x,\,y\in\Gamma_{\e_1}\cap\Omega\qquad\hbox{i.e.,}\qquad0<\rho(x),\,\rho(y)<\e_1,
\]
and
\[
d-\e_1<X\ls d-\rho(x)-\rho(y).
\]
Thus, $X>d/2$ and
\[
\rho(x)+\rho(y)\ls d-X<\e_1.
\]
Note that $\nabla\log\rho=\nabla\rho/\rho$, it is easy to see that
the left-hand side of \eqref{ineq-13-12-18-22-2} is less than
\begin{eqnarray}\label{ineq-13-12-19-15-1}
&&\kappa\Big[\frac{1}{\rho(y)}\cdot\nabla\rho(y)
-\frac{1}{\rho(x)}\cdot\nabla\rho(x)\Big]\cdot\tilde{X}\nonumber\\
&=&-\kappa\Big\{\frac{1}{\rho(x)}\cdot\big(\nabla\rho(x)\cdot\tilde{X}\big)
+\frac{1}{\rho(y)}\cdot\big[\nabla\rho(y)\cdot(-\tilde{X})\big]\Big\}
\end{eqnarray}

On the other hand, Lemma \ref{lem-13-12-31-56} implies that
\[
\nabla\rho(x)\cdot\tilde{X}\gs\theta_0\qquad\hbox{and}\qquad
\nabla\rho(y)\cdot(-\tilde{X})\gs\theta_0.
\]
Taking $\kappa=\kappa(n,\Omega):=8/\theta_0$
and using \eqref{phi-1st-eigenfunc-est}, we thus derive that
\begin{eqnarray*}
\Big[\nabla\log\rho^\kappa(y)-\nabla\log\rho^\kappa(x)\Big]\cdot\tilde{X}
&\ls&-\kappa\theta_0\Big[\frac{1}{\rho(x)}+\frac{1}{\rho(y)}\Big]\\
&<&-\frac{8}{\rho(x)+\rho(y)}\ls-\frac{8}{d-X}\\
&\ls&\Phi(X)+c_*-\frac{4}{d-X}\\
&\ls&\Phi(X)+c_*-\frac{4}{\e_1}<\Phi(X).
\end{eqnarray*}
So, the inequality \eqref{ineq-13-12-18-22-2} holds for this case.

The second possibility is that if $0<X\ls d-\e_1$, then
\begin{equation}\label{ineq-13a-1-8-10-8}
\pm\Phi(X)\ls c_1X,
\end{equation}
with $c_0=c_0(n,\Omega,\tilde{V},\tilde{\phi}_1)$.
Indeed, we derive from \eqref{eqn-14-5-23a-13-22} that
\[
\|\Phi'\|_{C([0,d-\e_1])}\ls\|\tilde{V}\|_{C([0,\frac{d}{2}])}+\tilde{\lambda}_1
+\|\tilde{\phi}'_1/\tilde{\phi}_1\|^2_{C([0,(d-\e_1)/2])}=:c_0.
\]
In addition, since $\tilde{\phi}_1$ is an even function,
then $\tilde{\phi}'_1(0)=0$. Thus $\Phi(0)=0$. Therefore,
\[
\pm\Phi(X)\ls\|\Phi'\|_{C([0,d-\e_1])}\cdot X\ls c_0X.
\]

On the other hand, we have
\begin{equation}\label{eq-13-11-29-22-15}
\bsqu{\nabla\log\rho^\kappa(y)-\nabla\log\rho^\kappa(x)}\cdot\tilde{X}\\
=\frac{\kappa}{X}\int_0^1\nabla^2\log\rho\big|_{ty+(1-t)x}(y-x,y-x)\,\mathrm{d}t.
\end{equation}
It is well known that the distance function $\rho(x)$ is concave
for a convex domain (since it is as the infimum for a family of planes),
hence $\nabla^2\log\rho\ls0$ in $\Omega$. Then \eqref{eq-13-11-29-22-15} together with this fact yield that
\begin{equation}\label{inq-13a-1-8-10-13}
\bsqu{\nabla\log\rho^\kappa(y)-\nabla\log\rho^\kappa(x)}\cdot\tilde{X}\ls0,
\end{equation}
Finally, we can derive the desired estimate from \eqref{ineq-13a-1-8-10-8}
and \eqref{inq-13a-1-8-10-13} in this case. So far we complete the proof.
\end{proof}

\begin{rem}
From the proof of Proposition \ref{prop-13a-1-2-12-1}, it is easy to see that
\[
\limsup_{x\rightarrow\p\Omega\atop(\textrm{or}\,\,y\rightarrow\p\Omega)}
\Big\{\bsqu{\nabla\log\rho^\kappa(y)-\nabla\log\rho^\kappa(x)}\cdot\tilde{X}-\Phi(X)-c_0X\Big\}\ls0,
\]
holds for every $x\neq y$ in $\Omega$.
\end{rem}

As a consequence we will construct an auxiliary function $u_0(x)$
(which is first due to Lu \cite{Lu-Seminar-May-26-2011}), so that
$\log\tilde{\phi}_1$ is a modulus of concavity of $u_0(x)$. This result
has already been used to prove Theorem \ref{sharp-log-concavity-estimate} in Section 3.
\begin{cor}\label{cor-13a-1-1-7-33}
Let $\Omega$ and $d$ be as in Lemma \ref{lem-13-12-31-56}. For $x\in\Omega$, let us define a function
\begin{equation}\label{initial-data-function}
u_0(x):=e^{-c_0\abs{x}^2/2}\rho^\kappa(x),
\end{equation}
where $c_0=c_0(n,\Omega,\tilde{V},\tilde{\phi}_1)$ and $\kappa=\kappa(n,\Omega)$ are as in Proposition \ref{prop-13a-1-2-12-1}.
Then $u_0$ satisfies
\begin{equation}\label{ineq-13a-1-2-12-16}
\bsqu{\nabla\log u_0(y)-\nabla\log u_0(x)}\cdot\tilde{X}\ls\Phi(X),
\end{equation}
for every $x\neq y$ in $\Omega$. Here $\Phi(s)$ is defined as before.
\end{cor}
\begin{proof}
We compute directly,
\begin{eqnarray*}
\nabla\log u_0(x)&=&\nabla\bsqu{-\frac{c_0}{2}\abs{x}^2+\log\rho^\kappa(x)}\\
&=&-c_0\abs{x}\cdot\frac{x}{\abs{x}}+\nabla\log\rho^\kappa(x)\\
&=&-c_0x+\nabla\log\rho^\kappa(x),
\end{eqnarray*}
and
\begin{eqnarray}\label{eq-14-6-11-22-33}
&&\bsqu{\nabla\log u_0(y)-\nabla\log u_0(x)}\cdot\tilde{X}\nonumber\\
&=&\big[-c_0y+\nabla\log\rho^\kappa(y)+c_0x-\nabla\log\rho^\kappa(x)\big]\cdot\tilde{X}\nonumber\\
&=&-c_0(y-x)\cdot\tilde{X}+\bsqu{\nabla\log\rho^\kappa(y)-\nabla\log\rho^\kappa(x)}\cdot\tilde{X}\nonumber\\
&=&-c_0X+\bsqu{\nabla\log\rho^\kappa(y)-\nabla\log\rho^\kappa(x)}\cdot\tilde{X}.
\end{eqnarray}
Here we used
\[
(y-x)\cdot\tilde{X}=(y-x)\cdot\frac{y-x}{\abs{y-x}}
=\abs{y-x}=X.
\]
Finally, we derive \eqref{ineq-13a-1-2-12-16} from \eqref{eq-14-6-11-22-33}
and \eqref{ineq-13-12-18-22-2} immediately. The proof is complete.
\end{proof}

\begin{rem}
Similarly, function $u_0$ is also of the following boundary asymptotic property:
\[
\limsup_{x\rightarrow\p\Omega\atop(\textrm{or}\,\,y\rightarrow\p\Omega)}\Big\{
\big[\nabla\log u_0(y)-\nabla\log u_0(x)\big]\cdot\tilde{X}-\Phi(X)\Big\}\ls0,
\]
holds for every $x\neq y$ in $\Omega$.
\end{rem}

\subsection*{Acknowledgment}
The author is deeply enlightened by Professor Zhi-Qin Lu's research talks \cite{Lu-Seminar-May-26-2011},
so he is extremely grateful to Professor Lu for the valuable ideas and methods in
\cite{Lu-Seminar-May-26-2011}, also the interesting discussions concerning eigenvalue gap theorems.


\bigbreak

\noindent{\sc Yue He},
Institute of Mathematics, School of Mathematics Sciences,
Nanjing Normal University, Nanjing, 210023, China\\
e-mail: heyueyn@163.com \& heyue@njnu.edu.cn


\end{document}